\documentclass{amsart}
\usepackage{array,geometry,amsmath,amsthm,amssymb,graphicx,tikz,tikz-cd,mathrsfs,fancyhdr,listings}
\usepackage[thinlines]{easytable}
\usepackage{bigstrut}
\usepackage{perpage}
\MakePerPage{footnote}

\usepackage{hyperref}
\hypersetup{colorlinks=true,citecolor=purple,linkcolor=purple}
\usepackage{ZJSstyle}

\title{Optimal topological generators of $U(1)$}
\author{Zachary Stier}
\email{zstier@princeton.edu}
\address{1113 Frist Campus Center, Princeton, NJ 08544.}
\date{February 2020}

\definecolor{codegreen}{rgb}{0,0.6,0}
\definecolor{codegray}{rgb}{0.5,0.5,0.5}
\definecolor{codepurple}{rgb}{0.58,0,0.82}
\definecolor{codeblue}{rgb}{0,0,0.95}
\definecolor{backcolor}{rgb}{0.95,0.95,0.92}
\begin{document}

\maketitle

\begin{abstract}
Sarnak's golden mean conjecture states that $(m+1)d_\vf(m)\le1+\frac{2}{\sqrt{5}}$ for all integers $m\ge1$, where $\vf$ is the golden mean and $d_\gt$ is the discrepancy function for $m+1$ multiples of $\gt$ modulo 1. In this paper, we characterize the set $\cS$ of values $\gt$ that share this property, as well as the set $\cT$ of those with the property for some lower bound $m\ge M$. Remarkably, $\cS\mod1$ has only 16 elements, whereas $\cT$ is the set of $GL_2(\Z)$-transformations of $\vf$. 
\end{abstract}

\section{Introduction}
	The unitary group $U(1)$ is compact with an invariant measure, and which may be modeled as acting by rotation on the circle $S^1\subset\C$ taken to have length 1. It is well known in this model that $U(1)$ is (monogenically) topologically generated by a rotation by any irrational angle.\footnote{For a nonabelian consideration, see e.g. Parzanchevski--Sarnak \cite{PS}.} A natural question here is which of these topological generators is the {\em best}. To answer this inquiry, we introduce the following function: 
	\begin{definition-non}[cf. \cite{GV}]
		Let $[[m]]=\{0,\dots,m\}$.\footnote{This is in contrast to $[m]$, which denotes $\{1,\dots,m\}$.} Define $d_\gt(m)$ as 
		$$\sup\{\abs{I} : \text{$I\subset\R$ an interval},(I+\Z)\cap[[m]]\gt=\emptyset\}.$$
	\end{definition-non}
	$d_\gt(m)$ measures the largest ``gap,'' modulo 1, of $m+1$ consecutive integer multiples of the real number $\gt$. It is clear that if $\gt$ is rational with the reduced fraction representation $\gt=\frac{a}{b}$, then $d_\gt(m)=\frac{1}{b}$ for all $m\ge b-1$. Meanwhile, when $\gt$ is irrational it is a topological generator of $U(1)$, so
	$$\lim\limits_{m\to\infty}d_\gt(m)=0$$
	weakly monotonically. For all choices of $\gt$, $(m+1)d_\gt(m)\ge1$ since equality is attained precisely when $d_\gt(m)=\frac{1}{m+1}$, but by the pigeonhole principle, $d_\gt(m)\ge\frac{1}{m+1}$. Therefore, $d_\gt(m)$ can be thought of as the {\em discrepancy} between the first $m+1$ iterates of $\gt$ and an equidistribution, and $(m+1)d_\gt(m)$ can be thought of as measuring how quickly $d_\gt(m)$ tends to 0 for irrational $\gt$. 
	
	Graham and van Lint \cite{GV} studied asymptotic behavior of this quantity, using the language of continued fractions. We say that two continued fractions $\gt$ and $\gs$ are {\em equivalent}, written $\gt\asymp\gs$, if there are positive integers $m$ and $n$ such that $\gt$ and $\gs$ agree after removing the length-$m$ and length-$n$ prefixes, respectively. The golden ratio is $\vf=\frac{1+\sqrt{5}}{2}$, and has continued fraction consisting of all 1's. 
	\begin{theorem-non}[\cite{GV}, Theorem 2]
		For any irrational $\gt$, 
		$$\limsup\limits_{m\to\infty}(m+1)d_\gt(m)\ge1+\frac{2}{\sqrt{5}}$$
		with equality iff $\gt\asymp\vf$. 
	\end{theorem-non}
	Here, we prove a stronger result about these asymptotics: 
	\begin{theorem}\label{under nec}
		Given $\gt\in\R$, there exists $M\in\N$ for which $m\ge M$ implies $(m+1)d_\gt(m)<1+\frac{2}{\sqrt{5}}$ if and only if $\gt\asymp\vf$. 
	\end{theorem}
	Letting $\cT$ be the set of values $\gt$ for which the condition on $d_\gt(m)$ in \thmref{under nec} holds, we will see, as is well known, that $\cT$ is the set of linear fractional transformations by $GL_2(\Z)$ of $\vf$, a dense countable subset of $\R$. 
	
	For many choices of $\gt$, $(m+1)d_\gt(m)$ rises above $1+\frac{2}{\sqrt{5}}$ before settling below, i.e. $M=1$ as in \thmref{under nec} does not suffice for us here. To study this new sought-after phenomenon---a global generalization of $\limsup\limits_{m\to\infty}(m+1)d_\gt(m)=1+\frac{2}{\sqrt{5}}$---we introduce a new measure of quality for topological generators. 
	\begin{definition-non}
		$D(\gt)=\sup\limits_{m\in\N}(m+1)d_m(\gt)$.
	\end{definition-non}
	From \cite{GV}, $D(\gt)\ge1+\frac{2}{\sqrt{5}}$ with equality on some (possibly empty) subset $\cS\subset\cT$. Sarnak conjectured, and Mozzochi recently proved, the following (the ``golden mean conjecture''): 
	\begin{theorem-non}[\cite{Moz}]
		$D(\vf)=1+\frac{2}{\sqrt{5}}$. 
	\end{theorem-non}
	This can be expanded to a surprising result completely characterizing $\cS$. 
	\begin{theorem}\label{under suff}
		There exist exactly 16 values $\gt$, modulo 1, for which $D(\gt)=1+\frac{2}{\sqrt{5}}$, which are specified in \figref{unders}.\footnote{$d_\gt(m)=d_{1-\gt}(m)$, so $\gt\in\cS$ if and only if $1-\gt\in\cS$, which is why only 8 values are specified in the table.}
	\end{theorem}	
	Unsurprisingly, $\vf$ (and $\vf^2=\vf+1$) is in one of these 16 modulo-1 classes: note that $\vf+\eta_7=2$. 
	
	One way to measure the ``quality'' of a generator on $1\le m\le M$ is by the largest value of $(m+1)d_\gt(m)$ attained on that range. To put this formally, we introduce: 
	\begin{definition-non}
		$D_M(\gt)=\max\limits_{m\in[M]}(m+1)d_\gt(m)$. 
	\end{definition-non}
	Then, there is no single ``best'' generator, in the sense of minimizing this quantity: 
	\begin{theorem}\label{no best}
		For each $\gt_0\in\cS$, there are infinitely many values $M\in\N$ for which $\gt_0=\argmin\limits_{\gt\in\cS}D_M(\gt)$. 
	\end{theorem}
	
	\begin{figure}
		\centering
		\begin{tabular}{|r|cccccc|c|c|c|}
				& 0 & 1 & 2 & 3 & 4 & 5 & matrix & exact & num.~val.\\\hline
			$\eta_{7}$ & 0 & 2 & 1 & 1 & 1 & $\dot1$ & $\begin{pmatrix}1 \\ 2 & 1\end{pmatrix}$ & $\frac{3-\sqrt{5}}{2}$ & 0.381\dots \\\hline
			$\eta_{6}$ & 0 & 2 & 1 & 2 & 1 & $\dot1$ & $\begin{pmatrix}3 & 1 \\ 8 & 3\end{pmatrix}$ & $\frac{25-\sqrt{5}}{62}$ & 0.367\dots \\\hline
			$\eta_{8}$ & 0 & 2 & 2 & 1 & 1 & $\dot1$ & $\begin{pmatrix}2 & 1 \\ 5 & 2\end{pmatrix}$ & $\frac{7+\sqrt{5}}{22}$ & 0.419\dots \\\hline
			$\eta_{4}$ & 0 & 3 & 1 & 1 & 1 & $\dot1$ & $\begin{pmatrix}1 \\ 3 & 1\end{pmatrix}$ & $\frac{5-\sqrt{5}}{10}$ & 0.276\dots \\\hline
			$\eta_{5}$ & 0 & 3 & 2 & 1 & 1 & $\dot1$ & $\begin{pmatrix}2 & 1 \\ 7 & 3\end{pmatrix}$ & $\frac{9+\sqrt{5}}{38}$ & 0.295\dots \\\hline
			$\eta_{2}$ & 0 & 4 & 1 & 1 & 1 & $\dot1$ & $\begin{pmatrix}1 \\ 4 & 1\end{pmatrix}$ & $\frac{7-\sqrt{5}}{22}$ & 0.216\dots \\\hline
			$\eta_{3}$ & 0 & 4 & 2 & 1 & 1 & $\dot1$ & $\begin{pmatrix}2 & 1 \\ 9 & 4\end{pmatrix}$ & $\frac{11+\sqrt{5}}{58}$ & 0.228\dots \\\hline
			$\eta_{1}$ & 0 & 5 & 2 & 1 & 1 & $\dot1$ & $\begin{pmatrix}2 & 1 \\ 11 & 5\end{pmatrix}$ & $\frac{13+\sqrt{5}}{82}$ & 0.185\dots \\\hline
		\end{tabular}
		\caption{The values in $(\cS\mod1)\cap\lbr{0,\half}$ in lexicographic order of continued fraction. Indices reflect the canonical order with respect to embedding the $\cS\mod1\hookrightarrow[0,1]$ in the obvious way.}
		\label{unders}
	\end{figure}

\section{Definitions and past results}
	Henceforth let $\gt$ be irrational. $d_\gt(m)$ may be evaluated exactly, using the language of continued fractions. We recall the following from \cite{GV,HW}: 
	\begin{definition-non}
		Consider the infinite continued fraction $\gt=[a_0,a_1,\dots]$.\footnote{It is elementary that $\gt$ must have a continued fraction and that it cannot be finite.} We have the following notation, for nonnegative integers $n$: 
		\begin{itemize}
			\item $\frac{h_n}{k_n}=\frac{a_nh_{n-1}+h_{n-2}}{a_nk_{n-1}+k_{n-2}}=[a_0,a_1,\dots,a_n]$ is the $n$th convergent. 
			\item $x_n=[a_{n+1},\dots,a_1]$. 
			\item $\gt_n=[a_n,a_{n+1},\dots]$. 
			\item $[a_0,\dots,a_{n-1},\dot1]=[a_0,\dots,a_{n-1},1,1,1,\dots]$. 
		\end{itemize}
	\end{definition-non}
	\begin{remark}\label{rmk}
		Let $\gt=[a_0,\dots,a_N,\dot{1}]$, where for $n>N$ we have $a_n=1$. Then, for such $n=N+d$, $k_n=F_{d+1}k_N+F_dk_{N-1}$. By the recurrence $k_n=a_nk_{n-1}+k_{n-2}$ and the stipulation that $a_n\in\N$, $k_n\ge F_{n+1}$.
	\end{remark}
	Indeed, the $n$th convergent $g_n=[1,\dots,1]$ to $\vf=[\dot1]$ equals $\frac{F_{n+2}}{F_{n+1}}$, for $F_n$ the $n$th Fibonacci number, indexed from $F_0=0$ and $F_1=1$, and so in this way $\vf$ has the smallest convergents. 

	Using our new notation, we can write more concisely that if $\gt\asymp\gs$ then there exist positive integers $m$ and $n$ for which $\gt_m=\gs_n$. The relationship between equivalent continued fractions can be made even more explicit: 
	
	\begin{theorem-non}[cf. \cite{HW}, Theorems 174 and 176]
		Equivalence of continued fractions is an equivalence relation, and two continued fractions $\gt$ and $\gs$ are equivalent if and only if there exists $\fM=\begin{pmatrix}a&b\\c&d\end{pmatrix}\in GL_2(\Z)$ for which $\gt=\frac{a\gs+b}{c\gs+d}$, denoted by $\fM\gs$ in this context. 
	\end{theorem-non}
	In this terminology, the aforementioned theorem of \cite{GV} and \thmref{under nec} can be thought of as a biconditional with $\gt\asymp\vf$, and $\cT$ can be seen as $GL_2(\Z)\vf$. 
	
	The following are long-established results about continued fractions: 
	\begin{lemma}[cf. \cite{HW}, pp.140]
		Fixing again $\frac{h_n}{k_n}$ and $\gt_n$ with respect to $\gt$: 
		\begin{align}
			\gt-\frac{h_n}{k_n}&=\frac{(-1)^n}{k_n(k_{n-1}+k_n\gt_{n+1})}.\label{diff w convergent}\tag{$*$}
		\end{align}
	\end{lemma}
	With these notions in hand, the following is proved by Slater \cite{Sla} and S\'os \cite{Sos} and used extensively in \cite{GV}: 
	\begin{lemma}
		Given $\gt=[a_0,a_1,\dots]$ and nonnegative integers $\ga$ and $m$ satisfying $\ga<a_{n+2}$ and $k_n+(\ga+1)k_{n+1}-1\le m\le k_n+(\ga+2)k_{n+1}-2$, it is the case that
		\begin{equation}
			d_\gt(m)=\abs{(k_n\gt-h_n)-\ga(k_{n+1}\gt-h_{n+1})}.\label{gv lem 1}\tag{$**$}
		\end{equation}
	\end{lemma}
	Combining \eqref{diff w convergent} into \eqref{gv lem 1}, with some algebraic manipulation we have:
	\begin{cor}\label{cor:useful}
		Given $\gt=[a_0,a_1,\dots]$ and nonnegative integers $\ga$ and $m$ satisfying $\ga<a_{n+2}$ and $k_n+(\ga+1)k_{n+1}-1\le m\le k_n+(\ga+2)k_{n+1}-2$, it is the case that
		$$d_\gt(m)=\frac{\gt_{n+2}-\ga}{k_n+k_{n+1}\gt_{n+2}}.$$
	\end{cor}
	Henceforth, let $\rho=1+\frac{2}{\sqrt{5}}$. 

\section{Proofs of \thmref{under nec} and \thmref{under suff}}\label{main}
	
	\thmref{under nec} asserts that $\cT=GL_2(\Z)\vf$, via linear fractional transformation; that is, $\cT$ is the set of continued fractions $\gt\asymp\vf$. Towards the proof of this result, we first prove a useful lemma. Of course, this lemma can be generalized considerably, but this is not needed to prove the result in mind. 

	\begin{lemma}\label{monotone}
		Let $f(x)=[1,\dots,1,x]$ be a function on $\R^+$, where the continued fraction is length $n+2$. Then $f$ is monotonic (either increasing or decreasing). 
	\end{lemma}
	\begin{proof}
		Fix $x$. Then 
		$$f(x)=\frac{h_nx+h_{n-1}}{k_nx+k_{n-1}}=\frac{F_{n+2}x+F_{n+1}}{F_{n+1}x+F_n}$$
		which is clearly differentiable on $\R^+$, so taking the derivative gives 
		$$f'(x)=\frac{F_nF_{n+2}-F_{n+1}^2}{(F_{n+1}x+F_n)^2}=\pm\frac{1}{(F_{n+1}x+F_n)^2}$$
		which has constant sign in $x$. 
	\end{proof}
	This simple lemma equips us to characterize the set $\cT$. 
	\begin{proof}[Proof of \thmref{under nec}]
		We know from \cite{GV} that equivalence to $\vf$ is necessary, since if $\gt\not\asymp\vf$ then $\limsup(m+1)d_\gt(m)=\ell>\rho$ so for all $M_0\in\N$, there is $m>M_0$ with $(m+1)d_\gt(m)>\half(\ell+\rho)>\rho$. 
		
		We now show that equivalence to $\vf$ is sufficient. Write $\gt=[0,a_1\dots,a_N,\dot{1}]$, where for $n>N$ we have $a_n=1$. \cite{Moz} shows that when $a_n=1$ and $k_n+k_{n+1}-1\le m\le k_n+2k_{n+1}-2$, 
		$$\max(m+1)d_\gt(m)=\frac{1+2x_n-\frac{1}{k_n}}{\gt_{n+1}+\frac{1}{x_{n-1}}}=\frac{1+2x_n-\frac{1}{k_n}}{\vf+x_n-1}.$$
		We see that 
		\begin{equation}
			x_n<\vf+\frac{5+2\sqrt{5}}{k_n}\label{eq:thm}\tag{$\star$}
		\end{equation}
		is necessary and sufficient to show $\max(m+1)d_\gt(m)<\rho$ over that range for $m$, by algebraic manipulation. $x_n=[1,\dots,1,a_N,\dots,a_1]$ with $d=n-N$ 1's. By \lemref{monotone} and since $a_N\in\N$ implies $a_N\ge1$, $x_n$ is bounded between $g_{d+1}$ and $g_d$, so 
		$$\abs{\vf-x_n}\le\max\lcr{\abs{\vf-g_d},\abs{\vf-g_{d+1}}}<\frac{1}{F_{d+1}^2}.$$
		Since $k_n=F_{d+1}k_N+F_dk_{N-1}$, we simply require $F_{d+1}^2>\frac{F_{d+1}k_N+F_dk_{N-1}}{5+2\sqrt{5}}$. This holds if 
		\begin{equation}
			F_{d+1}>\frac{k_N+k_{N-1}}{5+2\sqrt{5}}\label{ineq:thm}\tag{$\star\star$}
		\end{equation}
		which, since $d$ is variable while $N$ is fixed, is eventually true. If we let $d_0$ be the least $d$ for which \eqref{ineq:thm} holds, and let $N_0=N+d_0$, then we see that \eqref{eq:thm} holds for $n\ge N_0$ and so the theorem holds for $M_0=k_{N_0}+k_{N_0+1}-1$. 
	\end{proof}

	We now investigate when the lower bound can be made $M_0=1$, and we let $\cS$ denote the set of such irrational numbers. Of course, by \thmref{under nec}, any such generator is equivalent to $\vf$. While $\cT$ is dense in $\R$, \thmref{under suff} asserts that $\cS$ is remarkably sparse: $\#(\cS\mod1)=16$. Towards this result, we prove two lemmas. The first establishes when the continued fractions of $\cS$'s elements must become $\dot1$. The second establishes upper bounds on the values that can appear in the prefix of those continued fractions. It is then merely a matter of verifying with the aid of a short computer program (\secref{python app}) which values suffice. 

	\begin{lemma}\label{when ones}
		Write $\gt=[0,a_1,\dots]$. Suppose $n\ge6$. If $a_n>1$ then $\gt\not\in\cS$. 
	\end{lemma}
	\begin{proof}
		If $\gt\not\in GL_2(\z)\vf$ then we already know the result to hold, by \thmref{under nec}. So, we take $\gt\in GL_2(\z)\vf$. 
	
		Suppose towards contradiction that for some $N\ge5$, $a_n=1$ for all $n\ge N+2$, but $a_{N+1}>1$, yet $\gt\in\cS$. From \corref{cor:useful}, we have for $k_N+k_{N+1}-1\le m\le k_N+2k_{N+1}-2$: 
		$$d_\gt(m)=\frac{1}{\gt_{N+1}k_N+k_{N-1}}.$$
		It therefore follows that for $m=k_N+2k_{N+1}-2$: 
		$$(m+1)d_\gt(m)=\frac{2k_{N+1}+k_N-1}{\gt_{N+1}k_N+k_{N-1}}.$$
		Since $\gt\in\cS$, $(m+1)d_\gt(m)<\rho$. Rearranging the inequality, along with the substitutions 
		\begin{align*}
			\gt_{N+1}&=a_{N+1}-1+\vf\\
			k_{N+1}&=a_{N+1}k_N+k_{N-1},
		\end{align*}
		yields the following:
		$$((2-\rho)a_{N+1}+1+\rho-\rho\vf)k_N+(2-\rho)k_{N-1}<1.$$
		Using \rmkref{rmk}, the fact that $a_{N+1}\ge2$ by hypothesis, and numerical values of $\vf$ and $\rho$, we note that the left-hand side is lower-bounded by $0.04F_{N+1}+0.1F_N$, which, since $F_6=13$ and $F_5=8$, is lower-bounded by 1.3. This provides the desired contradiction and proves the result. 
	\end{proof}

	\begin{lemma}\label{abcde bounds}
		If $[0,a,b,c,d,e,\dot{1}]\in\cS$, then: 
		\begin{align*}
			a&\le18, & b&\le18, & c&\le14, & d&\le12, & e&\le11.
		\end{align*}
	\end{lemma}
	\begin{proof}
		Consider any $\gt=[0,a_1,\dots]\in\cS$, and fix $n\in[5]$. We know that we have for $k_{n-1}+(\ga+1)k_n-1\le m\le k_{n-1}+(\ga+2)k_n-2$, $d_\gt(m)=\frac{1}{\gt_nk_{n-1}+k_{n-2}}$ and so $(m+1)d_\gt(m)$ attains its maximum on this range: 
		$$\frac{k_{n-1}+(\ga+2)k_n-1}{\gt_nk_{n-1}+k_{n-2}}.$$
		In order for this value to be less than $\rho$ (a necessary---but far from sufficient---condition for $\gt\in\cS$), we must have, for $\ga=0$:
		$$k_{n-1}+2k_n-1<\rho(\gt_nk_{n-1}+k_{n-2}).$$
		Using the substitutions
		\begin{align*}
			\gt_n&=a_n+\frac{1}{\gt_{n+1}}\\
			k_n&=a_nk_{n-1}+k_{n-2}
		\end{align*}
		we apply the fact that $\gt_{n+1}\ge1$ and rearrange to obtain 
		$$(2-\rho)k_{n-1}a_n+k_{n-1}\lpr{1-\frac{\rho}{\gt_{n+1}}}+(2-\rho)k_{n-2}<1$$
		and therefore 
		\begin{align*}
			a_n&<\frac{\rho-1}{2-\rho}+\frac{k_{n-2}}{k_{n-1}}+\frac{1}{(2-\rho)k_{n-1}}\\
				 &<\frac{\rho-1}{2-\rho}+1+\frac{1}{(2-\rho)F_n}.
		\end{align*}
		Using the numerical value of $\rho$ and letting $n$ range on [5] gives the desired bounds. 
	\end{proof}

	\begin{proof}[Proof of \thmref{under suff}]
		\lemref{when ones} and \lemref{abcde bounds} are sufficient to prove that $\#(S\mod1)<\infty$. Running the code specified in \secref{python app} reveals the values specified in \figref{unders}. All that remains to be shown is the correctness of the program; each step is evident except for why \texttt{n} only needs to be checked up to 29. This is merely a consequence of \eqref{ineq:thm} for $N=5$, specifically in the ``worst case'' (in terms of the sizes of $k_4$ and $k_5$) of $\gt=[0,18,18,14,12,11,\dot{1}]$, where $k_4=55141$ and $k_5=611119$ so $F_{d+1}>\frac{k_5+k_4}{5+2\sqrt{5}}\approx70000$, hence $d=24$. Because this justifies the code used, the Theorem is true. 
	\end{proof}
	To demonstrate the empirical difference between $\cS$ and a worse choice of $\gt$, see \figref{eye test} for the partition of the circle for $m=75$ for each element of $\cS$ as well as $\gt=\pi$. Stylistically, these diagrams are inspired by Motta, Shipman, and Springer's Figure 1 \cite{MSS}. When there are three distinct lengths, the longest one is colored red and the shortest green; when there are two distinct lengths (\figref{eye test}(d)), the longer one is colored orange and the shorter black. The code for this figure is found in \secref{mathematica app}. 
	
	\begin{figure}
		\begin{tabular}{ccc}
		\includegraphics[width=0.25\textwidth]{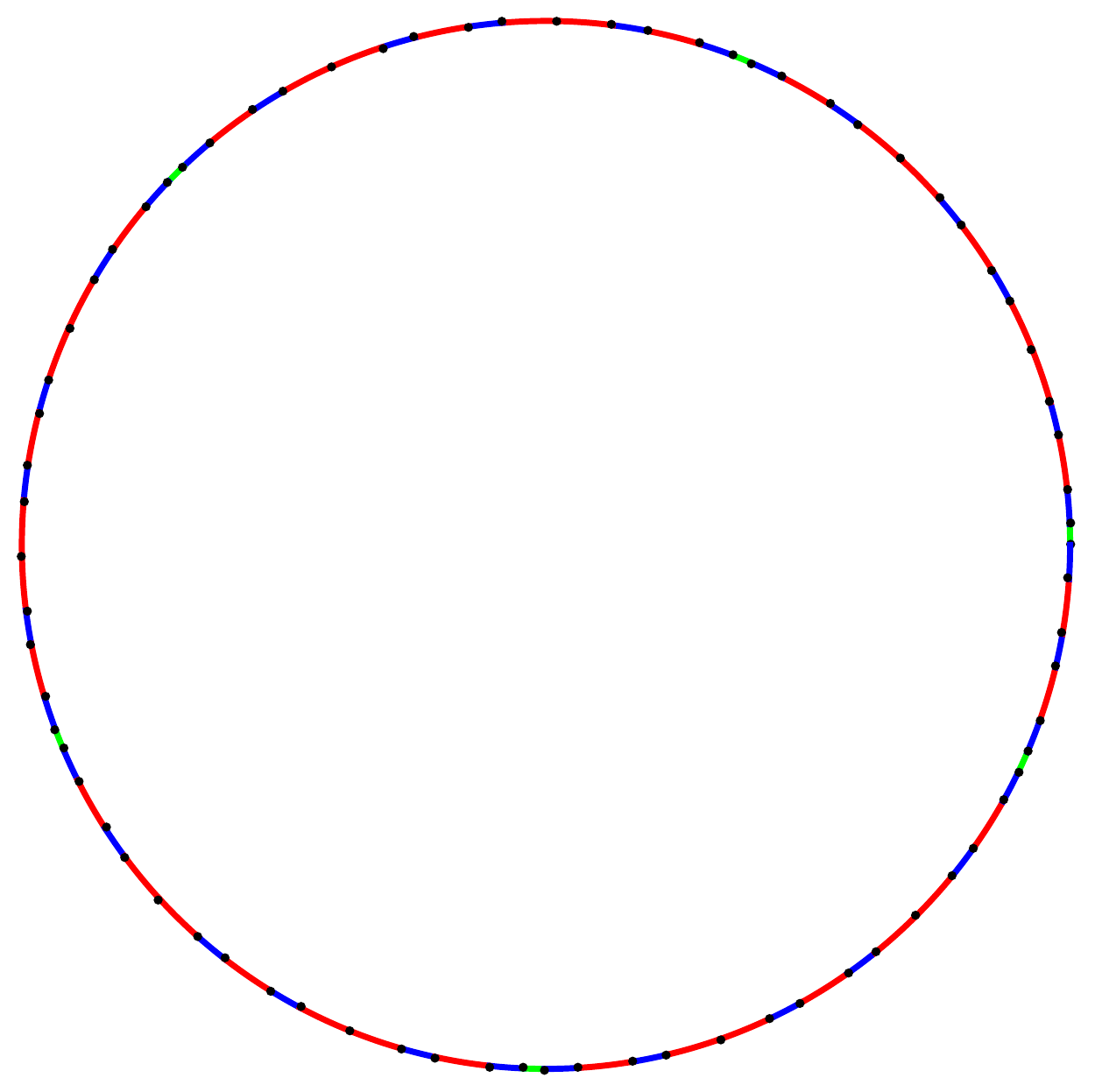} & \includegraphics[width=0.25\textwidth]{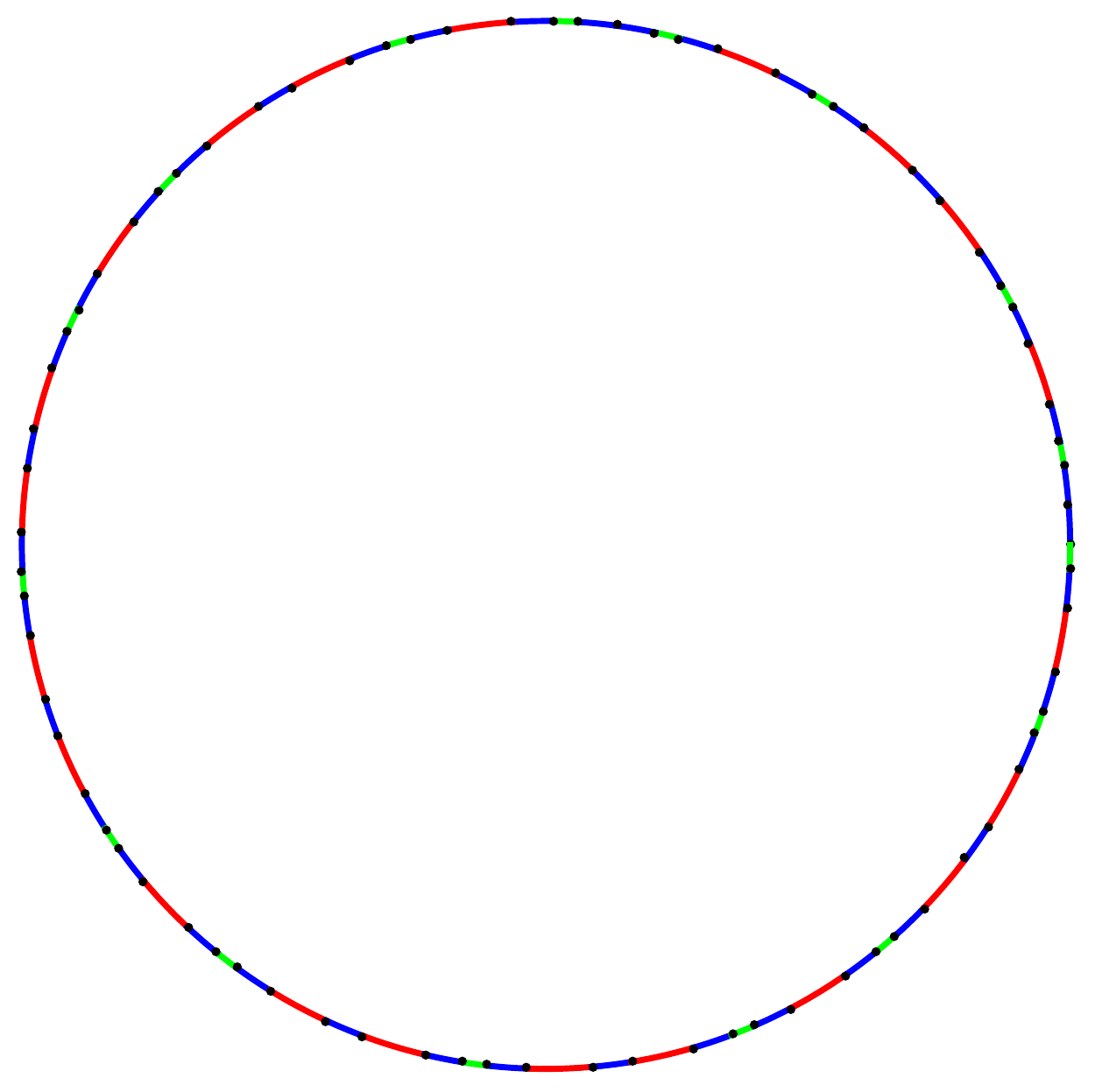} & \includegraphics[width=0.25\textwidth]{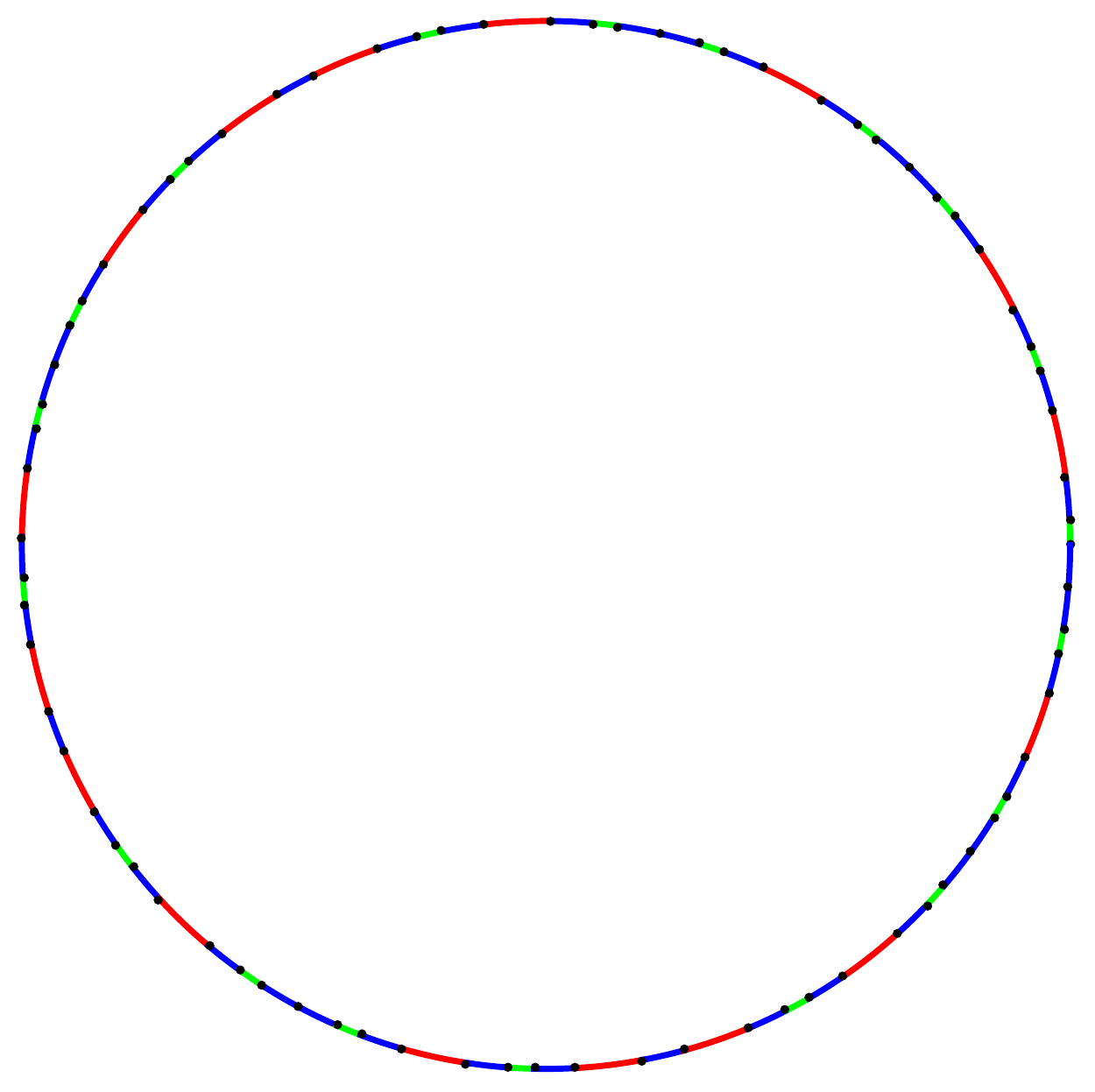} \\
			(a) $\eta_1$ & (b) $\eta_2$ & (c) $\eta_3$ \bigstrut[b] \\
		\includegraphics[width=0.25\textwidth]{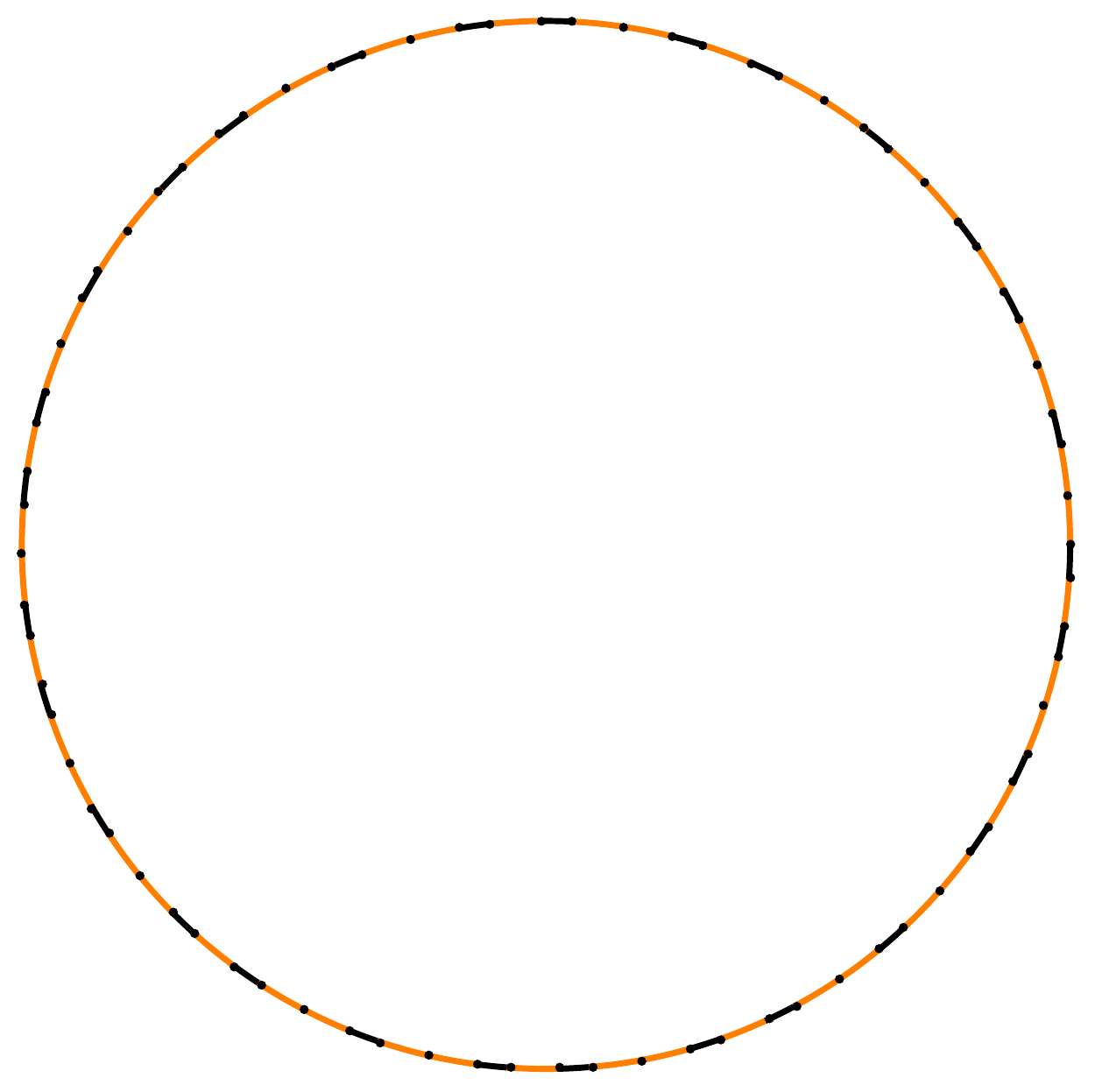} & \includegraphics[width=0.25\textwidth]{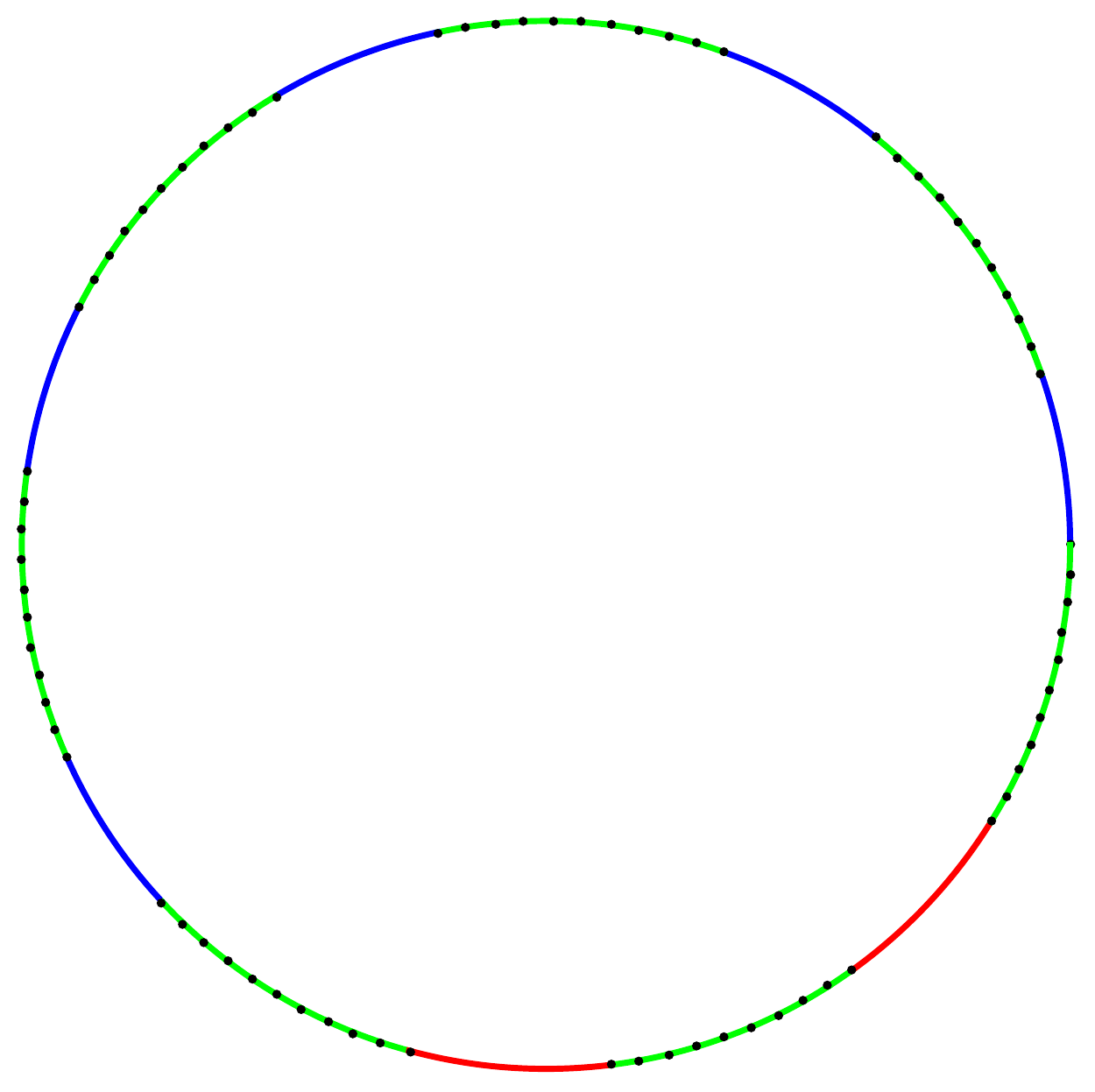} & \includegraphics[width=0.25\textwidth]{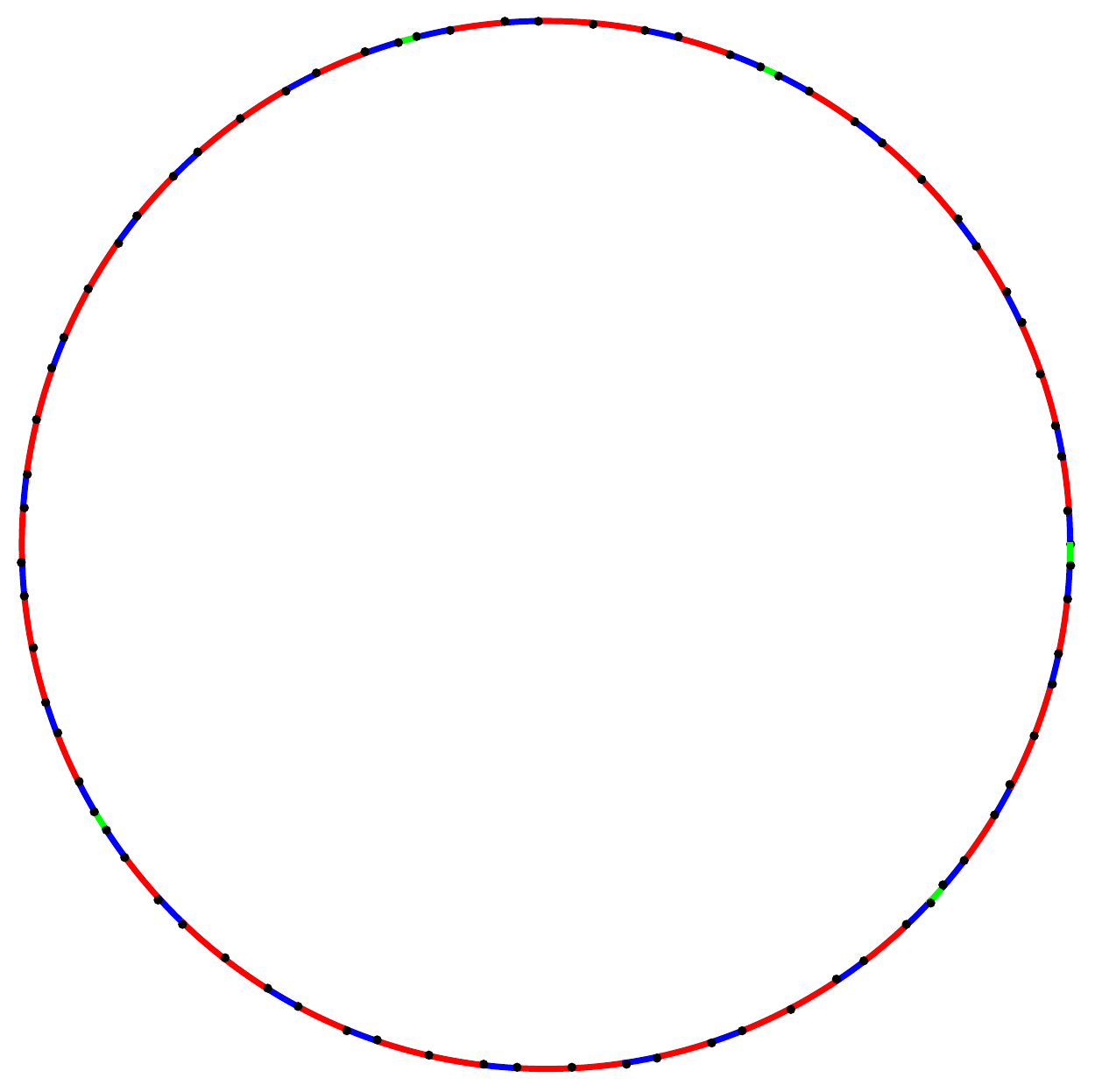} \\
			(d) $\eta_4$ & (e) $\pi$ & (f) $\eta_5$ \bigstrut[b] \\
		\includegraphics[width=0.25\textwidth]{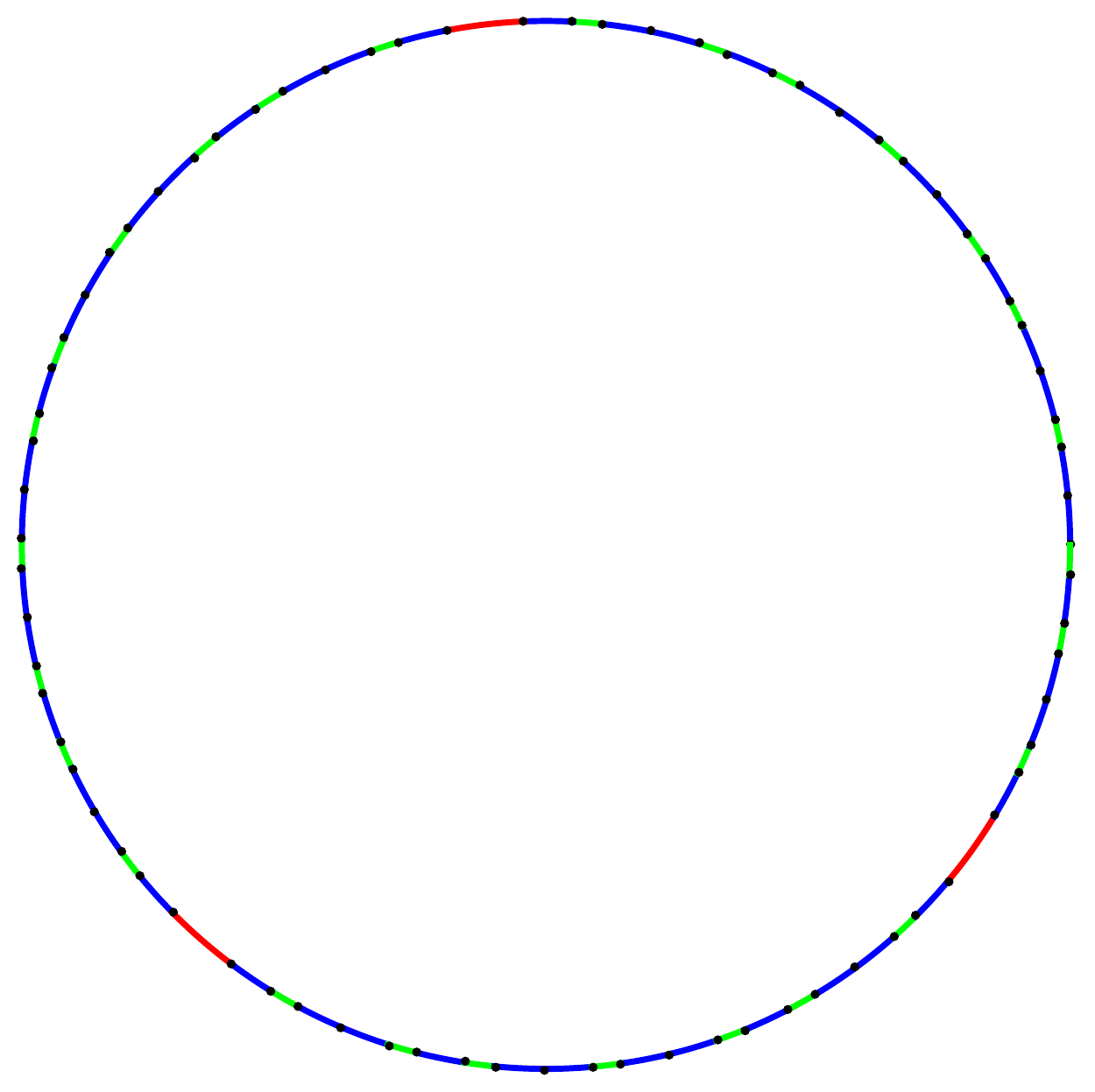} & \includegraphics[width=0.25\textwidth]{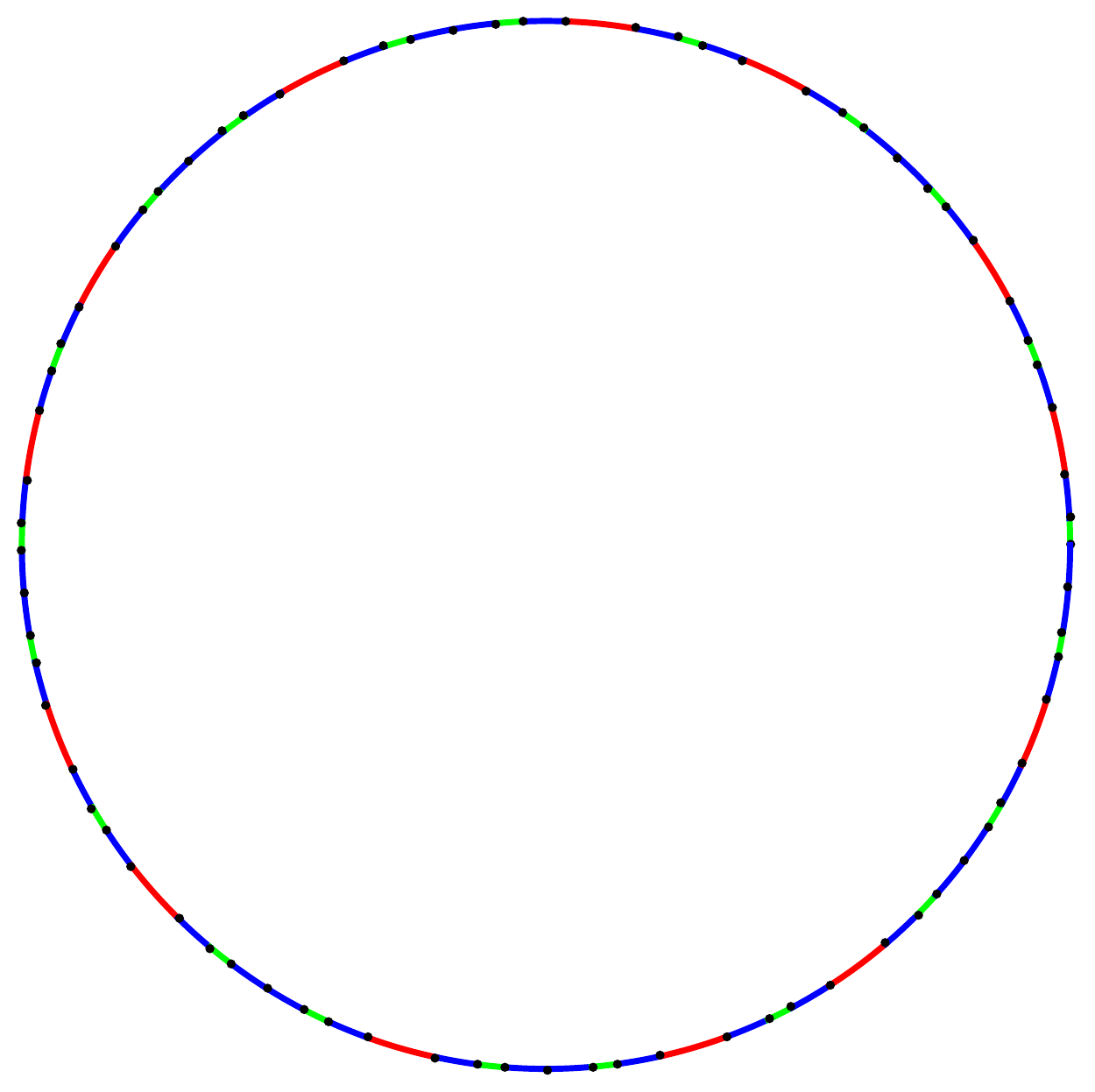} & \includegraphics[width=0.25\textwidth]{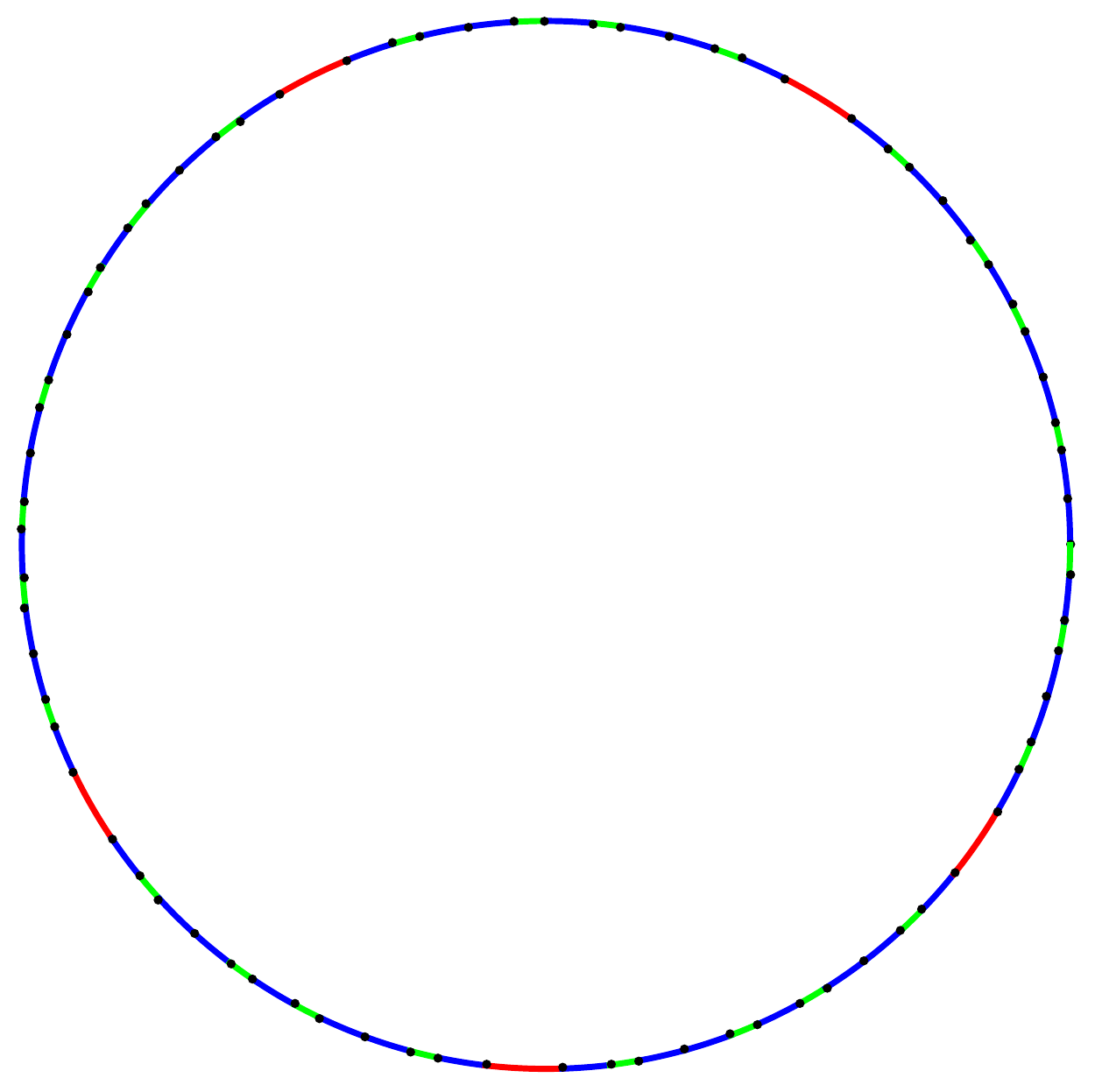} \\
			(g) $\eta_6$ & (h) $\eta_7$ & (i) $\eta_8$
		\end{tabular}
		\caption{The partition of $S^1$ for nine values of $\gt$ with $m=75$. Note that for (e), the partition is far less uniform than in the other figures.}
		\label{eye test}
	\end{figure}	
	
\section{Proof of \thmref{no best}}
	
	\begin{remark-non}
		Let $\bowtie$ be the equivalence relation on $\cS$ of $\gt\bowtie\gu$ iff $\gt\pm\gu\in\Z$. Clearly $\#(\cS/\bowtie)=8$, and for $\gt,\gu\in\cS$, $d_\gt(m)=d_\gu(m)$ iff $\gt\bowtie\gu$. Therefore, $f_m:(\cS/\bowtie)\to\R^+$ with $f_m(\gt)=d_\gt(m)$ is well-defined. $\cS/\bowtie$ has the convenient choice of representatives $\{\eta_i:i\in[8]\}$. 
	\end{remark-non}
	
	As a consequence of this remark, we treat $\cS$ implicitly as $\cS/\bowtie$ because of our primary concern with the context of $d_\gt(m)$. We now introduce some further notation. 
	\begin{definition-non}
		Define the functions $w:\N\to\cS$ and $W:\cS\times\N\to\R$ as
		\begin{align*}
			w(M)&=\argmin\limits_{\gt\in\cS} D_M(\gt)\\
			W_\gt(M)&=\#\{m\in[M]:\gt=w(m)\}. 
		\end{align*}
		We have the shorthand
		\begin{align*}
			LI(i)&=\liminf\limits_{M\to\infty}\frac{W_{\eta_i}(M)}{M}\\
			LS(i)&=\limsup\limits_{M\to\infty}\frac{W_{\eta_i}(M)}{M}.
		\end{align*}
	\end{definition-non}
	We now begin our approach towards \thmref{no best}. It is an immediate corollary to the following: 
	\begin{theorem}\label{help 3}
		We have the following asymptotics, where the third and fifth column the give the percentages rounded to the nearest tenth: 
		$$
		\begin{array}{c||c|r||c|r||}
			i & \multicolumn{2}{c||}{LI(i)} & \multicolumn{2}{c||}{LS(i)} \\\hline\hline
			1 & \frac{4\sqrt{5}-6}{11} & 26.8 &\frac{13+\sqrt{5}}{41} & 37.2	\bigstrut \\\hline 
			2 & \frac{7-2\sqrt{5}}{29} & 8.7 & \frac{2-3\sqrt{5}}{11} & 13.4 \bigstrut \\ \hline 
			3 & 9-4\sqrt{5} & 5.6 & \frac{7-2\sqrt{5}}{29} & 8.7 \bigstrut \\\hline 
			4 & \frac{11-3\sqrt{5}}{38} & 11.3 & \frac{3\sqrt{5}-5}{10} & 17.1 \bigstrut \\ \hline 
			5 & \frac{19-8\sqrt{5}}{41} & 2.7 & \frac{12-5\sqrt{5}}{19} & 4.3 \bigstrut \\ \hline 
			6 & \frac{7\sqrt{5}-15}{10} & 6.5 & \frac{13-3\sqrt{5}}{62} & 10.1 \bigstrut \\ \hline 
			7 & \frac{4-\sqrt{5}}{11} & 16.0 & \sqrt{5}-2 & 23.6 \bigstrut \\\hline 
			8 & \frac{27-11\sqrt{5}}{62} & 3.9 & \frac{17-7\sqrt{5}}{22} & 6.1 \bigstrut \\ \hline 
		\end{array}
		$$
		
		In particular, each of the $\liminf$s is positive. 
	\end{theorem}
	
	As an illustration of the alternating nature for small $M$, see \figref{small M}, where if $\eta_i=\argmin\limits_{\gt\in\cS}D_M(\gt)$ then the $M$th data point $\lpr{M,\min\limits_{\gt\in\cS}D_M(\gt)}$ is colored with the $i$th color in the following list: red, orange, purple, green, blue, brown, black, aquamarine. The code used to generate this figure can be found in \secref{code:small M}. 
	
	\begin{figure}
		\centering
		\includegraphics[width=0.75\textwidth]{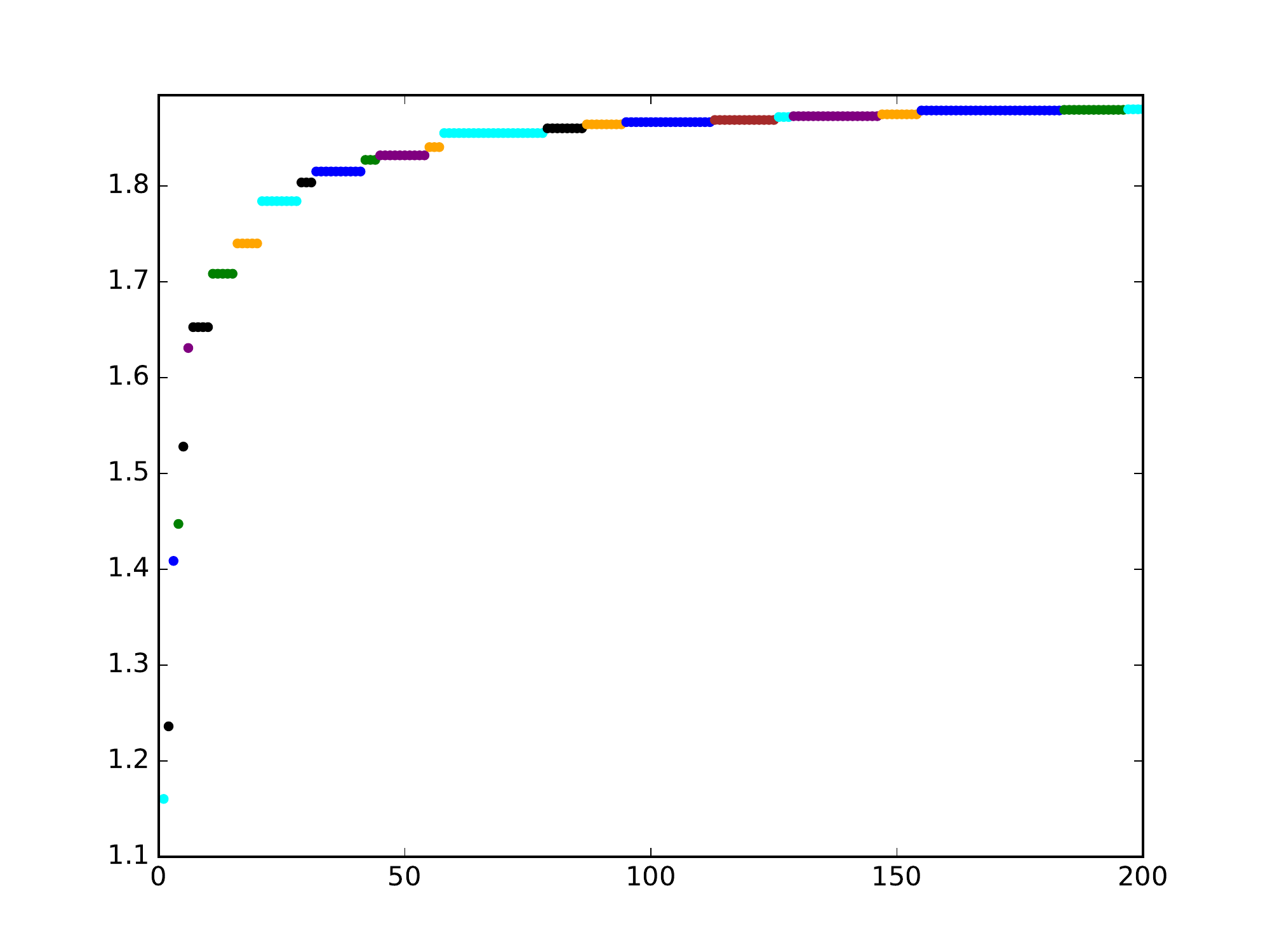}
		\caption{A plot of $\min\limits_{\gt\in\cS}D_M(\gt)$ for $M\in[200]$, colored corresponding to $\argmin\limits_{\gt\in\cS}D_M(\gt)$.}
		\label{small M}
	\end{figure}
	
	The proof of this Theorem involves indirectly computing particular values of $W$ by computing the values at which each $\gt\in\cS$ is the minimizer, in terms of the convergents. It is now convenient to look at the convergents as functions $k_5,k_6:\cS\to\N$: 
	\begin{align*}
		k_5(\eta_1) &= 43 & k_6(\eta_1) &= 70 \\
		k_5(\eta_2) &= 23 & k_6(\eta_2) &= 37 \\
		k_5(\eta_3) &= 35 & k_6(\eta_3) &= 57 \\
		k_5(\eta_4) &= 18 & k_6(\eta_4) &= 29 \\
		k_5(\eta_5) &= 27 & k_6(\eta_5) &= 44 \\
		k_5(\eta_6) &= 19 & k_6(\eta_6) &= 30 \\
		k_5(\eta_7) &= 13 & k_6(\eta_7) &= 21 \\
		k_5(\eta_8) &= 19 & k_6(\eta_8) &= 31 \\
	\end{align*}
	We then define new sequences 
	\begin{align*}
		K_n(1) &= 70F_n+43F_{n-1} \\ 
		K_n(2) &= 71F_n+44F_{n-1} \\ 
		K_n(3) &= 76F_n+47F_{n-1} \\ 
		K_n(4) &= 79F_n+49F_{n-1} \\ 
		K_n(5) &= 81F_n+50F_{n-1} \\ 
		K_n(6) &= 89F_n+55F_{n-1} \\ 
		K_n(7) &= 92F_n+57F_{n-1} \\ 
		K_n(8) &= 97F_n+60F_{n-1}		
	\end{align*}
	with the further convention that for any $n\in\N$, $m\in\Z$, and $i\in[8]$,
	$$K_n(i+8m)=K_{n-m}(i).$$
	So, for instance, $K_n(0)=K_{n-1}(8)$. 
	
	Note that this is merely a reindexing of each $k_n(\cdot)$ by the permutation $\pi=(2 \ 8 \ 5)(3 \ 7 \ 6 \ 4)\in S_8$ (that is, $K_n(j)$ is a shift of the convergents $k_{n}(\eta_i)$ for $j=\pi i$). Call $\copi=\pi\inv$. 
	\begin{lemma}\label{cycle}
		For all positive integers $n$, 
		$$K_n(1)<K_n(2)<K_n(3)<K_n(4)<K_n(5)<K_n(6)<K_n(7)<K_n(8)<K_{n+1}(1).$$
	\end{lemma}
	\begin{proof}
		Equivalently, $K_n(i)<K_n(j)<K_{n+1}(i)$ for all $1\le i<j\le 8$. The first inequality is obvious: if $K_n(i)=a_iF_n+b_iF_{n-1}$, then by inspection, $a_i<a_j$ whenever $i<j$. The second inequality comes from observing that $K_{n+1}(i)=a_iF_{n+1}+b_iF_n=(a_i+b_i)F_n+a_iF_{n-1}$ and since $a_i>b_j$ for all $i,j$. 
	\end{proof}
	\begin{lemma}\label{sigma and tau formulas}
		Define the sequences $\gs_n(i)$ and $\tau_n(i)$, where $\gs_n(i)<\tau_n(i)<\gs_{n+1}(i)-1$, as follows: 
		$$\{M\in\N : \eta_i=w(M)\}=\bigsqcup\limits_{n\in\N}[\gs_n(i),\tau_n(i)].$$
		Then, we have that $j=\pi i$ and
		\begin{align*}
			\gs_n(i)&=\ceil{(K_{n+3}(j-2)-3)\lpr{\frac{K_{n-1}(j-1)+K_n(j-1)\vf}{K_{n-1}(j-2)+K_n(j-2)\vf}}}-1\\
			\tau_n(i)&=\ceil{(K_{n+3}(j-1)-3)\lpr{\frac{K_{n-1}(j)+K_n(j)\vf}{K_{n-1}(j-1)+K_n(j-1)\vf}}}-2.
		\end{align*}
	\end{lemma}
	\begin{proof}
		We first establish that these sequences are well-defined for all $i$. 
		$w(M)$ is $\eta_i$ for which $D_M(\eta_i)<D_M(\eta_j)$ for all $j\neq i$. However, for all choices of $i\neq j$ and $n$, with $M_n(i)=K_n(\pi i)+2K_{n+1}(\pi i)-2=K_{n+3}(\pi i)-2$, we have 
		$$D_{M_n(i)}(\eta_i)\ge(M_n(i)+1)d_{\eta_i}(M_n(i))>D_{M_n(i)}(\eta_j).$$
		The first inequality is trivial. 
		The second follows by considering $m+1$ and $d_{\eta_j}(m)$ separately: clearly on $m\in[M_n(i)]$, $m+1\le M_n(i)+1$. Then, say for fixed $j$ that $K_{n+2}(\pi j)-1\le M_n(i)\le K_{n+3}(\pi j)-2$. By \lemref{cycle}, $M_{n-1}(i)<K_{n+2}(\pi j)<M_n(i)$, from which we conclude that $K_n(\pi i)<K_n(\pi j)$. By \corref{cor:useful}, we have that 
		\begin{align*}
			d_{\eta_i}(M_n(i))&=\frac{1}{K_{n-1}(\pi i)+K_n(\pi i)\vf} \\
			d_{\eta_j}(M_n(i))&=\frac{1}{K_{n-1}(\pi j)+K_n(\pi j)\vf}
		\end{align*}
		Therefore $d_{\eta_i}(M_n(i))>d_{\eta_j}(M_n(i))$, concluding the second inequality. Thus, there are infinitely many values $M$ (e.g. those of the form $M_n(i)$) at which $\eta_i\neq w(M)$. Hence $\gs_n(i)$ and $\tau_n(i)$ are well-defined sequences for all $i$. 
		
		Further, it is evident from the above argument that the ``order of succession'' for $M$ sufficiently large, e.g. $M\ge K_1(1)=70$, is $\eta_{\copi i}$ for $i=1,2,\dots,8$ and repeating---that is, $w(M)=\eta_1$ for $M$ on some interval $[s_1,s_2-1]$, followed by $w(M)=\eta_2$ on $[s_2,s_3-1]$, etc., up to $w(M)=\eta_8$ on $[s_8,s_1'-1]$, and then this cycle repeats with $w(M)=\eta_1$ on $[s_1',s_2'-1]$. Therefore we just need to compare $\eta_{\copi i}$ against $\eta_{\copi(i-1)}$ and $\eta_{\copi(i+1)}$. See \figref{runs} for an illustration of the interval-based behvaior. 
		
		\begin{figure}
			\centering
			\includegraphics[width=0.75\textwidth]{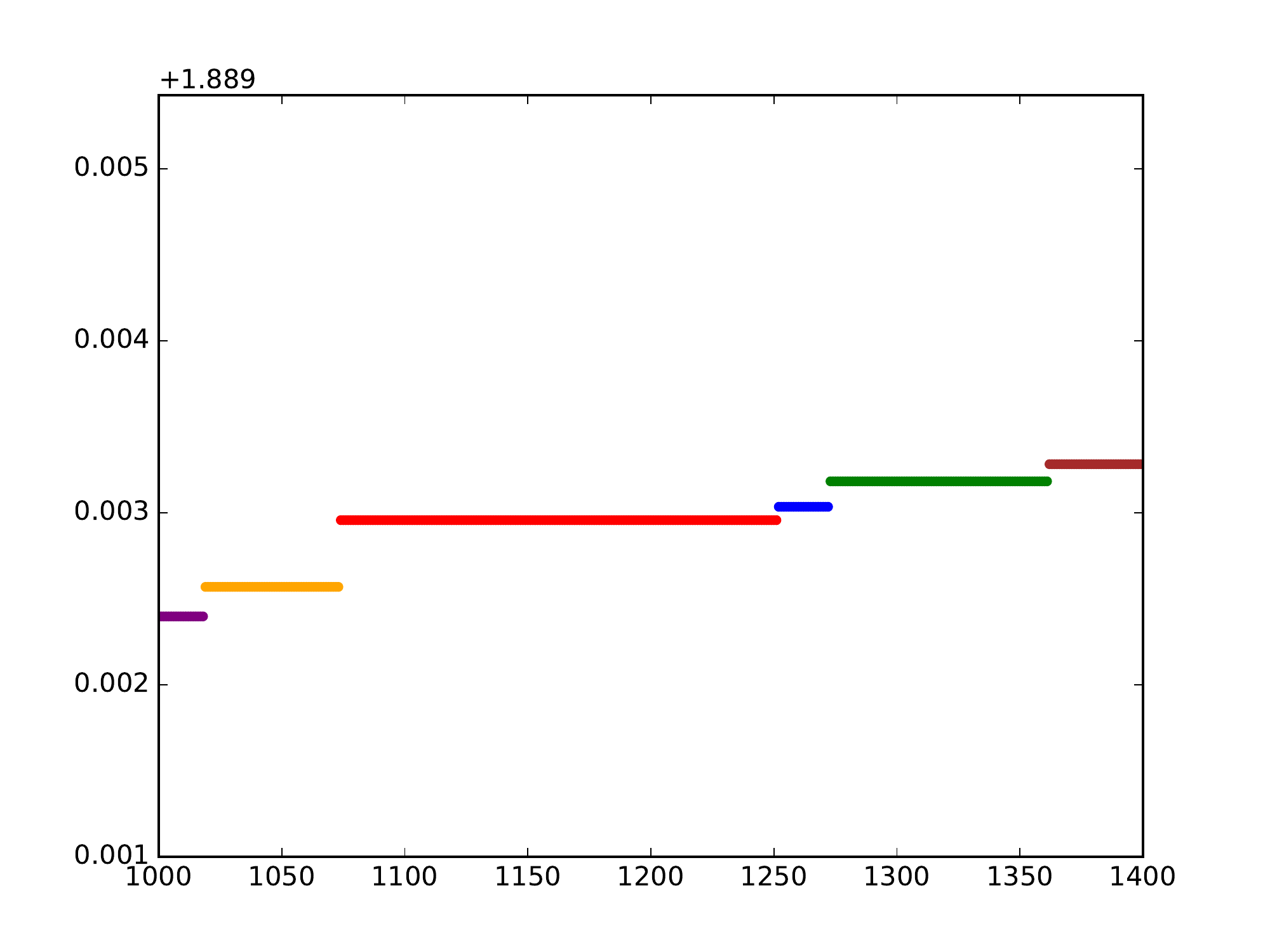}
			\caption{A plot of $\min\limits_{\gt\in\cS}D_M(\gt)$ for $M\in[1200,1400]$, colored corresponding to $\argmin\limits_{\gt\in\cS}D_M(\gt)$.}
			\label{runs}
		\end{figure}
		
		It is now convenient to define ``dual'' sequences $\hgs_n(i)$ and $\htau_n(i)$ defined as $[\hgs_n(i),\htau_n(i)]\ni M$ the $n$th range on $\N\cap[70,\infty)$ for which $\eta_{\copi i}$ {\em maximizes} $D_M(\gt)$ over $\gt\in\cS$. We see that for similar reasons, this maximizer cycles through $1,2,\dots,8$. We compute $\hgs_n(i)$ by considering $\eta_{\copi(i-1)},\eta_{\copi i}$: at what value $m>M_n(\copi(i-1))$ does it first occur that 
		$$(m+1)d_{\eta_{\copi i}}(m)\ge(M_n(\copi(i-1))+1)d_{\eta_{\copi(i-1)}}(M_n(\copi(i-1)))?$$
		Algebraic manipulation gives $m\ge(M_n(\copi(i-1))-1)\lpr{\frac{K_{n-1}(\copi i)+K_n(\copi i)\vf}{K_{n-1}(\copi(i-1))+K_n(\copi(i-1))\vf}}$, hence 
		$$\hgs_n(i)=\ceil{(M_n(\copi(i-1))-1)\lpr{\frac{K_{n-1}(\copi i)+K_n(\copi i)\vf}{K_{n-1}(\copi(i-1))+K_n(\copi(i-1))\vf}}-1}.$$
		Then, since $\N\cap[70,\infty)\subset\bigsqcup\limits_{n\in\N}\bigsqcup\limits_{i\in[8]}[\hgs_n(i),\htau_n(i)]$, we immediately obtain the relationship
		$$\htau_n(i)=\hgs_n(i+1)-1.$$
		Finally, we observe that 
		\begin{align*}
			\gs_n(i)&=\hgs_n(\pi i-1)\\
			\tau_n(i)&=\htau_n(\pi i-1)=\hgs_n(\pi i)-1
		\end{align*}
		because by that interval, all $j\neq i$ will have already achieved a maximum surpassing $\eta_{\copi i}$'s. 
	\end{proof}
	\begin{proof}[Proof of \thmref{help 3}]
		Given $i$ and $M$, let $j=\pi i$ and let $n$ be the greatest integer such that $K_n(j)\le M$. 
		$W_{\eta_i}(M)\in\gT\lpr{\tau_n(j)-\gs_n(j)}$ and so we have the following asymptotic tendencies: 
		\begin{align*}
			\liminf\limits_{M\to\infty}\frac{W_{\eta_i}(M)}{M}&=\lim\limits_{n\to\infty}\frac{\tau_n(j)-\gs_n(j)}{\gs_{n+1}(j)-\gs_n(j)} \\
			\limsup\limits_{M\to\infty}\frac{W_{\eta_i}(M)}{M}&=\lim\limits_{n\to\infty}\frac{\tau_n(j)-\gs_n(j)}{\tau_n(j)-\tau_{n-1}(j)}
		\end{align*}
		and using the exact values computed in \lemref{sigma and tau formulas} gives the stated values. 
	\end{proof}
	
	We can interpret this result as saying that as $M$ grows, each element of $\cS$ is represented as $w(M)$ infinitely many times. Further, $\eta_1=w(M)$ with marginally higher probability than the alternatives. 
	
	There is an interesting parallel to be drawn with Theorems \ref{no best} and \ref{help 3} and with work in analytic number theory on prime distributions. In 1914, Littlewood \cite{Lit} proved the unexpected fact that the difference $\pi(x)-\li(x)$ alternates infinitely often.\footnote{Here, $\pi(x,q,a)$ counts primes $p<x$ with $p\equiv a\pmod{q}$ with $\pi(x)$ implicitly having $(q,a)=(1,0)$ and $\li$ is the logarithmic integral $\int_0^x\frac{dt}{\log t}$.} Likewise, \thmref{no best} gives eightfold (rather than twofold) alternation. Earlier, in 1853, Chebyshev noticed that $\pi(x,4,3)>\pi(x,4,1)$ despite the asymptotic behavior $\frac{\pi(x,4,3)}{\pi(x,4,1)}\to1$, a result strengthened and generalized considerably by Rubinstein--Sarnak \cite{RS} and termed ``Chebbyshev's bias.'' Here we see a much stronger emergent bias in the statement of \thmref{help 3}, where there exists some $M_0\in\N$ where for all $M>M_0$, we have
	$$W_{\eta_1}(M)>W_{\eta_7}(M)>W_{\eta_4}(M)>W_{\eta_2}(M)>W_{\eta_6}(M)>W_{\eta_3}(M)>W_{\eta_8}(M)>W_{\eta_5}(M).$$
	
	In preliminary explorations that became this paper, an attempt was made at the related problem of 
	\begin{center}
		\bf for each $M\in[49]$, minimize $D_M(\gt)$ over all $\gt\in\lbr{0,\half}$. 
	\end{center}
	The approach was to na\"{i}vely sample from the interval a large number of times (100000) for each $M$. Except when $M$ takes the values 30 and 31---where the optimum is approximately $\frac{1}{30}$ and $\frac{1}{31}$, respectively, to within one part in $10^6$---the values agree with the problem constrained for $\gt\in\cS$ as is solved in this section of the text to within one part in at least $10^3$. 
	
	We can also compare these results with Ridley \cite{Rid}, which studies a related problem in packing efficiency of features in plants which grow at fixed divergence angles. There, the optimal angle (out of total angle 1) is determined to be $(\vf-1)^2$; note that $\eta_7=(\vf-1)^2$ (as enumerated in \figref{unders}). Therefore, we see that Ridley's notion of optimality coincides with the notion explored here using $D_M(\gt)$ when $M$ takes the values 2, 5, 7--10, 29, 45, and 47--49, where in Ridley's model, $M$ represents the number of generations, that is, the number of features (e.g. petals on a flower) that have grown using the constant divergence angle $\gt$. 
	
\section*{Acknowledgements}
	This work was completed as part of my senior thesis at Princeton University. I am grateful to my advisor Peter Sarnak for suggesting this problem and for his guidance throughout.

\section{Code}

	\subsection{Python 2.7 code for the proof of \thmref{under suff}}\label{python app}
		The following Python code was used following \lemref{abcde bounds} to prove \thmref{under suff}. 
		\begin{lstlisting}[language=Python]
phi = (1+5**0.5)/2
rho = 1+2/5**0.5
def V(n,a):
	xl = [(a*i) % 1 for i in range(n+1)]
	xl.sort()
	xl.append(1)
	maxGap = 0
	for i in range(n+1):
		if xl[i+1] - xl[i] > maxGap:
			maxGap = xl[i+1] - xl[i]
	return maxGap

A = 18
B = 18
C = 14
D = 12
E = 11

for a in range(1,A+1):
	for b in range(1,B+1):
		for c in range(1,C+1):
			for d in range(1,D+1):
				for e in range(1,E+1):
					gt = [1,a,b,c,d,e,1,1,1,1,1,1,1,1,1,1,1, \
							1,1,1,1,1,1,1,1,1,1,1,1,1,1,1,1]
					xp = 1+1/(a+1/(b+1/(c+1/(d+1/(e+1/phi)))))
					kcurr = 0
					knext = 1
					continueQ = True
					for n in range(-1,30):
						xp = 1./(xp-gt[n+1])
						if n > 6:
							xp = phi
						for ga in range(0,gt[n+2]):
							m = kcurr+(ga+2)*knext-2
							if (m+1)*(xp-ga)/(knext*xp+kcurr) > rho:
								continueQ = False
								break
						if not continueQ:
							break
						newknext = gt[n+2]*knext+kcurr
						kcurr = knext
						knext = newknext
					if continueQ:
						print a,b,c,d,e
\end{lstlisting}
		
	\subsection{Mathematica 12 code for generating \figref{eye test}}\label{mathematica app}
		The following Mathematica code was used to generate \figref{eye test}. 
		
		{\bf Warning:} due to internal precision error, the code sometimes crashes. The source of this error is in \texttt{pos = Sort[N[DeleteDuplicates[Differences[L] // FullSimplify]]];} where \texttt{DeleteDuplicates} might leave a list of length longer than 3, in turn causing \texttt{nearest3} to throw an error. This can be resolved manually for given \texttt{a} and \texttt{n}. 
\begin{lstlisting}[language=Mathematica]
nearest2[{a_, b_}][x_] := If[x < (a + b)/2, 1, 2];
nearest3[{a_, b_, c_}][x_] := If[x < (a + b)/2, 1, 
								If[(a + b)/2 <= x < (b + c)/2, 2, 3]];
tricolor[a_, n_] := Module[{L = Vlist[a, n], pos},
	pos = Sort[N[DeleteDuplicates[Differences[L] // FullSimplify]]];
	Table[{
		If[Length[pos] == 2, {Black, Orange}[[
			nearest2[pos][L[[i + 1]] - L[[i]]]]], {Green, Blue, Red}[[
			nearest3[pos][L[[i + 1]] - L[[i]]]]]],
		Thick, Circle[{0, 0}, 1, {2Pi L[[i]], 2Pi L[[i + 1]]}],
		Black, Point[{Cos[2Pi L[[i]]], Sin[2Pi L[[i]]]}]
		}, {i, 1, n + 1}]];

Manipulate[Graphics[tricolor[a, n]], {n, 1, 100, 1},
 {a, {(13 + Sqrt[5])/82 -> "\!\(\*SubscriptBox[\(a\), \(1\)]\)", 
		( 7 - Sqrt[5])/22 -> "\!\(\*SubscriptBox[\(a\), \(2\)]\)",
		(11 + Sqrt[5])/58 -> "\!\(\*SubscriptBox[\(a\), \(3\)]\)",
		( 5 - Sqrt[5])/10 -> "\!\(\*SubscriptBox[\(a\), \(4\)]\)",
		( 9 + Sqrt[5])/38 -> "\!\(\*SubscriptBox[\(a\), \(5\)]\)",
		(25 - Sqrt[5])/62 -> "\!\(\*SubscriptBox[\(a\), \(6\)]\)",
		( 3 - Sqrt[5])/ 2 -> "\!\(\*SubscriptBox[\(a\), \(7\)]\)",
		( 7 + Sqrt[5])/22 -> "\!\(\*SubscriptBox[\(a\), \(8\)]\)"}}]
\end{lstlisting}

	\subsection{Python 2.7 code for generating \figref{small M}}\label{code:small M}
		The following Python code was used to generate \figref{small M}. It is admittedly not the most efficient way to handle this data, but given the relatively small numbers used, ease of coding took priority over asymptotic efficiency. 
		
		\texttt{V} and \texttt{rho} are as in \secref{python app}. 
		
		In order to produce an output on a different range $[a,b]$ of $x$-axis values (such as in \figref{runs}), replace the outer loop with \texttt{for m in range(1,b+1)} and the last line with \texttt{plt.xlim(a,b)}. 
		\begin{lstlisting}[language=Python]
import matplotlib.pyplot as plt
from math import sqrt
etas = [ (13+sqrt(5))/82, (7-sqrt(5))/22, (11+sqrt(5))/58, (5-sqrt(5))/10, (9+sqrt(5))/38, (25-sqrt(5))/62, (3-sqrt(5))/2, (7+sqrt(5))/22 ]
colors = ["red","orange","purple","green","blue","brown","black","aqua"] 
Vs = [[],]*8

for m in range(1,201):
	min = rho
	minAt = 8
	for i in range(8):
		Vs[i] = Vs[i] + [(m+1)*V(m,etas[i])]
		if max(Vs[i]) < min:
			min = max(Vs[i])
			minAt = i
	plt.scatter([m],[min],c=colors[minAt],linewidths=0)
plt.ylim(top=rho)
plt.xlim(0,200)
\end{lstlisting}

\end{document}